\documentclass[]{article}
\usepackage{amssymb,amsmath,amsthm,bm,citesort,setspace,mathrsfs}
\usepackage{url,color,units}
\input cyracc.def

\usepackage[normalem]{ulem}

\DeclareMathOperator{\Var}{Var}
\DeclareMathOperator{\Cov}{Cov}

\newcommand{\R}{\mathbf{R}}

\newcommand{\cA}{A}
\newcommand{\<}{\langle}
\renewcommand{\>}{\rangle}

\newcommand{\be}{\begin{equation}}
\newcommand{\ee}{\end{equation}}

\newcommand{\lip}{\text{\rm Lip}}

\renewcommand{\P}{\mathrm{P}}
\newcommand{\E}{\mathrm{E}}

\renewcommand{\d}{{\rm d}}

\newcommand{\e}{{\rm e}}

\renewcommand{\ge}{\geqslant}
\renewcommand{\le}{\leqslant}

\author{Le Chen\\University of Utah \and Davar Khoshnevisan\\University of Utah
\and Kunwoo Kim\\University of Utah}
\title{Decorrelation of total mass via energy\thanks{
	Research supported in part by a Swiss Federal Fellowship (L.C.)
	and a grant from the United States' National Science Foundation  (D.K.; DMS-1307470)}}
\date{October 22, 2014}
\newtheorem{stat}{Statement}[section]

\newtheorem{corollary}[stat]{Corollary}
\newtheorem{theorem}[stat]{Theorem}
\newtheorem{lemma}[stat]{Lemma}
\theoremstyle{definition}

\numberwithin{equation}{section}
\begin{document}
\maketitle
\begin{abstract}
	The main result of this small note is a quantified version of
	the assertion that if $u$ and $v$ solve two nonlinear stochastic
	heat equations, and if the mutual energy between the initial states
	of the two stochastic PDEs is small, then the total masses of
	the two systems are nearly uncorrelated for a very long time.
	One of the consequences of this fact is that a stochastic heat equation
	with regular coefficients is a finite system if and only if the
	initial state is integrable.\\


\end{abstract}

\section{Introduction}

Consider a stochastic partial differential equation of the type
\begin{equation}\label{SHE}
	\frac{\partial}{\partial t} u_t(x) = \frac{\theta}{2}\frac{\partial^2}{\partial x^2}
	u_t(x) + \sigma(u_t(x))\xi(t\,,x),
\end{equation}
for all $t>0$ and $x\in\R$, where $\xi$ denotes space-time white noise;
that is, $\xi$ is a generalized centered Gaussian random field
with covariance measure
\begin{equation}
	\Cov[\xi(t\,,x)\,,\xi(s\,,y)]=\delta_0(s-t)\delta_0(x-y),
\end{equation}
for every $s,t\ge0$ and $x,y\in\R$.

Throughout, the diffusion coeffficient $\theta/2$ is assumed to be fixed,
finite, and strictly positive. In order to keep the discussion as non-technical
as possible we consider only the commonly-studied
setting in which the initial function $u_0$ is nonrandom and essentially
bounded, and the nonlinearity $\sigma:\R\to\R$ is nonrandom and
globally Lipschitz continuous. The theory of Walsh \cite{Walsh} ensures the existence
of a unique continuous solution $u$ to \eqref{SHE}.

In addition, we assume throughout that $\sigma$
satisfies the following conditions:
\begin{equation}\label{sigma}
	\sigma(x)\ge 0
	\quad\text{for all $x\in\R$ and}\quad \sigma(0)=0.
\end{equation}
Since $\sigma$ vanishes at the origin and $u_0(x)\ge0$ for all $x\in\R$,
it follows from a comparison theorem \cite{Mueller1,Mueller2} that
with probability one,
\begin{equation}\label{pos}
	u_t(x)\ge0\qquad
	\text{for all $t\ge0$ and $x\in\R$.}
\end{equation}
Because of this fact, and since $-\xi$ is also
a space-time white noise, the positivity condition on $\sigma$ is harmless.
Moreover, the positivity of $u$ suggests that we can think of $u_t(x)$
as the density  at location $x$ of a continuous particle system at time $t$.
The \emph{total mass} process of that particle system at
time $t\ge0$ is therefore
\begin{equation}\label{L1}
	\int_{-\infty}^\infty u_t(x)\,\d x = \|u_t\|_{L^1(\R)}.
\end{equation}
As a consequence,  $t\mapsto\|u_t\|_{L^1(\R)}$ is
a continuous local martingale as long as it is finite at all times.
This fact is basically due to Spitzer
\cite[Proposition 2.3]{Spitzer1981}---see also the proof
of Lemma \ref{lem:sufficient} below---and plays an important role
for example in Liggett's
analysis of linear particle systems \cite[Chapter IX]{Liggett}.

Motivated by the preceding, one says \cite[p.\ 432]{Liggett}
that the system is  \emph{finite}
if $\|u_t\|_{L^1(\R)}<\infty$
a.s.\ for all $t\ge0$ and that  it is \emph{infinite} if $\|u_t\|_{L^1(\R)}=\infty$
a.s.\ for all $t\ge0$.  Because of the way in which we have defined things,
it is logically possible that a particle system is neither finite nor infinite.
The following shows that this sort of anomaly cannot
occur in the context of \eqref{SHE}. Moreover,  that
there is a very simple characterization of when \eqref{SHE}
is a finite system.

\begin{corollary}\label{co:PAM}
	Let $u$ solve \eqref{SHE}, starting with a nonrandom
	and nonnegative initial function $u_0\in L^\infty(\R)$. Then,
	$u_0\in L^1(\R)$  if and only if the system \eqref{SHE} is finite.
\end{corollary}

This result is stated as a corollary of an inequality [Theorem \ref{th:main} below]
that says roughly that if $u$ and $v$ solve \eqref{SHE}---with respective nonlinearities
$\sigma_1$ and $\sigma_2$ and initial functions $u_0$ and $v_0$---and if the
``mutual energy'' between $u_0$ and $v_0$ is small, then the total mass of $u$ at
time $t$ is almost uncorrelated from the total mass of $v$ at time $t$
for a wide range of times $t$. This fact has other interesting consequences as well.
We name two of them next.

 The following result says that if $u_0$ and $v_0$ have bounded support
 and the support of $u_0$ is very far  away from the
 support of $v_0$, then the total mass of $u$ at time $t$ is almost uncorrelated
 from the total mass of $v$ at time $t$ for all large times $t$ upto a constant multiple
 of the distance between the supports of $u_0$ and $v_0$. In order to write
 this out more carefully, let $\mathcal{S}[f]$ denote the support of a function $f:\R\to\R$,
and let ``$\text{dist}$'' denote the Hausdorff distance between subsets of
the real line. Also, let $\lip(\varphi)$ to be the optimal Lipschitz constant
of a Lipschitz-continuous function $\varphi:\R\to\R$; that is,
\begin{equation}
	\lip(\varphi):= \sup_{-\infty<a<b<\infty}\left|
	\frac{\varphi(a)-\varphi(b)}{a-b}\right|.
\end{equation}
Then the preceding takes the following more precise form.

\begin{corollary}\label{co:1}
	Suppose $u$ and $v$ solve \eqref{SHE} with respective nonrandom
	nonnegative initial functions $u_0,v_0\in L^1(\R)\cap L^\infty(\R)$ that
	satisfy $\|u_0\|_{L^1(\R)}\wedge\|v_0\|_{L^1(\R)}>0$
	and  diffusion coefficients $\sigma_1,\sigma_2$
	that satisfy \eqref{sigma}. Then there exist
	$\eta\in(0\,,1)$ and $\gamma>0$---depending only on $(\lip(\sigma_1),
	\lip(\sigma_2),\theta)$---such that
	\begin{equation}
		\Cov\left( \frac{\|u_t\|_{L^1(\R)}}{\|u_0\|_{L^1(\R)}}\,,
		\frac{\|v_t\|_{L^1(\R)}}{\|v_0\|_{L^1(\R)}}\right)
		\le \eta^{-1} \e^{-\gamma t},
	\end{equation}
	provided that
	$0< t< \eta\cdot\text{\rm dist}  (\mathcal{S}[u_0]\,,\mathcal{S}[v_0] ).$
\end{corollary}

In order to motivate our third corollary, we first digress a little.
The Heisenberg uncertainty principle says, in one form or another,
that a nice function $\varphi$ and its Fourier transform
$\widehat\varphi$ cannot
both have small support; see Donoho and Stark \cite{DonohoStark}
for example.
Therefore one expects that, for many nice pairs  of
functions $u_0$ and $v_0$, $\text{\rm dist}(\mathcal{S}[u_0]\,,\mathcal{S}[v_0])$ ought to
be large if and only
if the Lebesgue measure of $\mathcal{S}[\widehat{u}_0]
\cap \mathcal{S}[\widehat{v}_0]$ is small. If this were true, then because of
Corollary \ref{co:1}, one would also expect that
if $\text{\rm meas}(\mathcal{S}[\widehat{u}_0]
\cap \mathcal{S}[\widehat{v}_0])$ were small, then $\|u_t\|_{L^1(\R)}$
and $\|v_t\|_{L^1(\R)}$  have small covariance for a long time.
The following shows that this last assertion is so, and also yields
a quantitative estimate of the ``how long.''

\begin{corollary}\label{co:2}
	Suppose $u$ and $v$ solve \eqref{SHE} with respective nonrandom
	nonnegative initial functions $u_0,v_0\in L^1(\R)\cap L^\infty(\R)$
	and diffusion coefficients $\sigma_1$ and $\sigma_2$ respectively.
	Then there exist
	$\eta\in(0\,,1)$ and $\gamma>0$---depending only on $\lip(\sigma_1)$,
	$\lip(\sigma_2)$, and $\theta$---such that
	\begin{equation}
		\Cov\left( \frac{\|u_t\|_{L^1(\R)}}{\|u_0\|_{L^1(\R)}}\,,
		\frac{\|v_t\|_{L^1(\R)}}{\|v_0\|_{L^1(\R)}}\right)
		\le \eta^{-1}\e^{-\gamma t},
	\end{equation}
	provided that
	$0< t < \eta\cdot\log_+ [1/\text{\rm meas} (
	\mathcal{S}[\widehat{u}_0]
	 \cap \mathcal{S}[\widehat{v}_0] ) ].$
\end{corollary}

Let us conclude the Introduction with a brief comment on
the previous corollary.

If the Fourier transforms of $u_0$ and $v_0$ were supported
on disjoint sets then Corollary \ref{co:2} would imply the extremely
surprising statement that
$\|u_t\|_{L^1(\R)}$ and $\|v_t\|_{L^1(\R)}$ are uncorrelated
for all $t\ge0$. It turns out that this condition of disjoint support typically does
not hold for reasons that are similar to the reasons behind Heisenberg's
uncertainty principle. In order to see this, let $\cA(\R)$ denote the
\emph{Wiener algebra};
that is, the collection of all $f\in L^1(\R)$ such that
$\widehat{f}\in L^1(\R)$. As usual, we identify two functions that are
a.e.\ equal. Owing to the Riemann--Lebesgue lemma,
$\cA(\R)\subset C_0(\R)$ [the latter being
the space of uniformly continuous functions on $\R$ that vanish at
infinity]. In particular, $\cA(\R)$ is dense in $L^p(\R)$ for all $p\in(0\,,\infty]$;
in turn, the Schwartz space $\mathscr{S}(\R)$ of test functions of
rapid decrease is dense in $\cA(\R)$. To summarize, $\cA(\R)$ is a {\it big}
collection of real functions.
Furthermore, the Fourier transform is a one-to-one and onto map from
$\cA(\R)$ to $\cA(\R)$ thanks to the inversion theorem.

Now suppose that
$u_0,v_0\in \cA(\R)$
are nonnegative functions that
have the additional property that their Fourier transforms have disjoint support.
By the inversion theorem, $\widehat{u}_0$ and $\widehat{v}_0$ are both
positive definite and continuous. They are,
in particular, maximized at the origin [theorem of Herglotz].
This implies that at least one of $u_0$ and $v_0$ is identically zero.

It is possible to appeal to the works of Chen and Dalang
\cite{ChenDalang13Heat} and Chen and Kim \cite{ChenKim14Comp} in order to
extend our results to the case that $u_0$ is a measure. For example,
that extension of Corollary \ref{co:PAM} asserts that the system is finite 
if and only if the initial measure is finite. See  
\S\ref{rem:measure} below for more details.

\section{Some background, and an inequality}
We follow the theory of Walsh \cite{Walsh}---see in particular, 
Dalang \cite{Dalang} and Chen and Dalang \cite{ChenDalang13Heat}---%
and interpret the SPDE \eqref{SHE} as the  random integral equation,
\begin{equation}\label{mild}
	u_t(x) =(p_t*u_0)(x) + (p\circledast\sigma(u))_t(x),
\end{equation}
where: (a)
$(p\circledast \Psi)_t(x) := \int_{(0,t)\times\R} p_{t-s}(x-y)
\Psi_s(y)\,\xi(\d s\,\d y)$
is a  stochastic integral, in the sense of Walsh  \cite{Walsh}, of
any predictable random field $\Psi$ that satisfies
$\int_0^t\d s\int_{-\infty}^\infty\d y\ [p_{t-s}(y-x)]^2\E[\Psi_s(y)^2]<\infty$;
and (b) $p$ denotes the heat kernel; i.e.,
\begin{equation}
	p_r(a) := (2\pi\theta r)^{-1/2}\exp\left( - \frac{a^2}{2\theta r}\right)\qquad
	[r>0,a\in\R].
\end{equation}

According to the theory of Dalang \cite{Dalang},
\begin{equation}
	\sup_{t\in[0,T]}\sup_{x\in\R}\E\left( |u_t(x)|^k\right)<\infty
	\qquad\text{for all finite $T>0$ and $k\ge 2$}.
\end{equation}
Moreover, the three quantities that appear naturally
on both sides of
\eqref{mild} are continuous functions of
$t>0$ and $x\in\R$  [up to a modification]. Therefore, there is a single null set off which
\eqref{mild} holds for all $t>0$ and $x\in\R$. We will refer to this fact
tacitly from now on.

Let us mention an elementary consequence of Theorem 9 of Dalang and Mueller
\cite{DalangMueller}.
The following is well known to experts; we include the short proof for
the sake of completeness.

\begin{lemma}\label{lem:sufficient}
	If $u_0\in L^1(\R)\cap L^\infty(\R)$, then
	$u_t\in L^1(\R)$ a.s.\ for all $t\ge0$, and
	\begin{equation}
		\sup_{t\in[0,T]}\E\left(\|u_t\|_{L^1(\R)}^2\right)<\infty
		\qquad\text{for all finite $T\ge0$.}
	\end{equation}
	In particular, $t\mapsto\|u_t\|_{L^1(\R)}$ is a nonnegative
	continuous $L^2(\Omega)$-martingale.
\end{lemma}

\begin{proof}
	On one hand, Theorem 2.4 of Chen and Dalang
	\cite{ChenDalang13Heat} implies, as part of the definition of a solution,
	that
	\begin{equation}
		\E\left(\int_0^t\d s\int_{-\infty}^\infty\d y\ [\sigma(u_s(y))]^2 \right)
		<\infty,
	\end{equation}
	for all $t>0$. On the other hand,
	the general theory of Walsh \cite{Walsh} implies that:
	\begin{enumerate}
	\item $t\mapsto \int_{(0,t)\times\R}
		\sigma(u_s(y))\,\xi(\d s\,\d y)$ defines a continuous $L^2$-martingale
		with quadratic variation $\int_0^t\d s\int_{-\infty}^\infty\d y\
		[\sigma(u_s(y))]^2$ at time $t>0$; and
	\item We have the stochastic Fubini theorem: Almost surely,
		\begin{equation}
			\int_{-\infty}^\infty (p\circledast\sigma(u))_t(x)\,\d x=
			\int_{(0,t)\times\R} \sigma(u_s(y))\,\xi(\d s\,\d y)
			\qquad\forall t>0.
		\end{equation}
	\end{enumerate}

	The preceding allows us to integrate both sides of \eqref{mild}
	$[\d x]$ in order to conclude that
	\begin{equation}
		\|u_t\|_{L^1(\R)} = \|u_0\|_{L^1(\R)} +
		\int_{(0,t)\times\R}\sigma(u_s(y))\,\xi(\d s\,\d y).
	\end{equation}
	The lemma follows. In fact, we also obtain the following
	estimate:
	\begin{equation}
		\E\left( \|u_t\|_{L^1(\R)}^2\right) \le \|u_0\|_{L^1(\R)}^2 +
		\lip^2\int_0^t \E\left( \|u_s\|_{L^2(\R)}^2\right)\,\d s,
	\end{equation}
	valid for all $t>0$.
\end{proof}

\begin{lemma}\label{lem:Cov}
	Suppose $u$ and $v$ are the solutions to \eqref{SHE} with 
	respect to diffusion coefficients $\sigma_1$ and $\sigma_2$ 
	respectively, and  nonrandom nonnegative initial functions 
	$u_0,v_0\in L^1(\R)\cap L^2(\R)$ respectively. Then,
	\begin{equation}
		\Cov\left( \|u_t\|_{L^1(\R)}\,,\|v_t\|_{L^1(\R)}\right)
		=\int_0^t\d s\int_{-\infty}^\infty\d y\
		\E\left[ \sigma_1(u_s(y))\sigma_2(v_s(y))\right],
	\end{equation}
	for all $t>0$. Thus, $\|u_t\|_{L^1(\R)}$
	and $\|v_t\|_{L^1(\R)}$ are positively correlated.
\end{lemma}

\begin{proof}
	It follows from \eqref{mild} that
	\begin{equation}\label{E(uv)}\begin{split}
		&\E\left[ u_t(x) v_t(x') \right]
			= (p_t*u_0)(x) (p_t*v_0)(x') \\
		&\quad + \int_0^t \d s
			\int_{-\infty}^\infty \d y\ p_{t-s}(y-x)p_{t-s}(y-x')
			\E\left[ \sigma_1(u_s(y))\sigma_2(v_s(y))\right]\d y,
	\end{split}\end{equation}
	for all $t>0$ and $x,x'\in\R$. Integrate $[\d x\,\d x']$ to finish.
\end{proof}

Before we continue our analysis of the covariance of the previous
lemma, let us define the functions
\begin{equation}\label{r}
	r_\beta(x) := \frac{1}{2\sqrt{\beta\theta}}\exp\left(-|x|\sqrt{
	\beta/\theta}\right)\qquad[\beta>0,x\in\R],
\end{equation}
and corresponding linear operators
\begin{equation}\label{R}
	(\mathcal{R}_\beta f)(x) := (r_\beta*f)(x)\qquad[\beta>0,x\in\R].
\end{equation}
It is not hard to see that $\{\mathcal{R}_\beta\}_{\beta>0}$
is the resolvent of the heat semigroup for Brownian motion
run at twice the standard speed. The corresponding Dirichlet
forms are denoted by
\begin{equation}\label{E}
	\mathcal{E}_\beta(f\,,g) :=
	\<f\,, \mathcal{R}_\beta g\>_{L^2(\R)},
\end{equation}
and $\mathcal{E}_\beta(f\,,g)$ is called the \emph{mutual $\beta$-energy between $u_0$
and $v_0$}, as is customary.
The next theorem is the main result of this paper. Its content
can be summarized loosely as follows: If
$\mathcal{E}_\beta(u_0\,,v_0)\ll \exp(-\beta t)$ for some $t,\beta>0$, then
$\|u_s\|_{L^1(\R)}$ and $\|v_s\|_{L^1(\R)}$ are nearly uncorrelated
for all $s\in(0\,,t)$. 

\begin{theorem}\label{th:main}
	For all $t\ge0$ and
	$\beta>[\lip(\sigma_1)\cdot\lip(\sigma_2)]^2/(4\theta)$,
	\begin{equation}
		\Cov\left( \|u_t\|_{L^1(\R)}\,,\|v_t\|_{L^1(\R)}\right)
		\le \frac{2\sqrt{\beta\theta}\,
		\lip(\sigma_1)\lip(\sigma_2)}{2\sqrt{\beta\theta} -
		 \lip(\sigma_1)\lip(\sigma_2)}\cdot\e^{\beta t}\mathcal{E}_\beta(u_0\,,v_0).
	\end{equation}
\end{theorem}

\begin{proof}
	Because of \eqref{E(uv)} and
	the positivity assertions \eqref{sigma} and \eqref{pos},
	\begin{align}
		\E\left[ u_t(x) v_t(x) \right]
			&\le(p_t*u_0)(x) (p_t*v_0)(x) \\\notag
		&\ + \lip(\sigma_1)\lip(\sigma_2)\int_0^t \d s
			\int_{-\infty}^\infty \d y\ \left[ p_{t-s}(y-x)\right]^2
			\E\left[ u_s(y) v_s(y)\right]\d y,
	\end{align}
	for every $t\ge0$ and $x\in\R$.
	We integrate both side $[\d x]$ and appeal to the Tonelli theorem to find that
	\begin{align}
		&\E\left[\< u_t\,,v_t\>_{L^2(\R)}\right] \\\notag
		&\le \left\< p_t*u_0\,, p_t*v_0
			\right\>_{L^2(\R)} + \lip(\sigma_1)\lip(\sigma_2)
			\int_0^t \|p_{t-s}\|_{L^2(\R)}^2\E\left[\<u_s\,,v_s \>_{L^2(\R)}\right]
			\d s,
	\end{align}
	Since $p_t*p_s=p_{t+s}$ and $\<p_t*f\,, g\>_{L^2(\R)}=\<f\,,p_t*g\>_{L^2(\R)}$,
	this means that
	\begin{align}\label{recur}
		&\E\left[\< u_t\,,v_t\>_{L^2(\R)}\right] \\\notag
		&\hskip.3in \le \left\< u_0\,, p_{2t}*v_0
			\right\>_{L^2(\R)} + \lip(\sigma_1)\lip(\sigma_2)
			\int_0^t p_{2(t-s)}(0)
			\E\left[ \<u_s\,,v_s\>_{L^2(\R)}\right]\d s.
	\end{align}

	Define, for all $\beta>0$,
	\begin{equation}
		Z(\beta) := \int_0^\infty\e^{-\beta t}
		\E\left[\<u_t\,,v_t\>_{L^2(\R)}\right]\d t.
	\end{equation}
	Of course, $Z(\beta)\ge 0$ because of
	\eqref{pos} and the Tonelli theorem.
	In addition, we can infer from \eqref{recur} and Lemma \ref{lem:Cov} that
	\begin{equation}\label{Cov:Z}\begin{split}
		\Cov\left( \|u_t\|_{L^1(\R)}\,, \|v_t\|_{L^1(\R)} \right)
			&\le \lip(\sigma_1) \lip(\sigma_2)\cdot\int_0^t
			\E\left[\< u_s\,,v_s\>_{L^2(\R)}\right]\d s\\
		&\le \e^{\beta t}\lip(\sigma_1) \lip(\sigma_2)\cdot Z(\beta),
	\end{split}\end{equation}
	valid for every $t,\beta>0$.

	It is possible to see that $Z(\beta)<\infty$ for $\beta$ large.
	Here is a crude proof: Apply the Cauchy--Schwarz inequality twice
	in order to see that
	\begin{equation}
		Z(\beta)
		\le \int_0^\infty \e^{-\beta t}\sqrt{\E\left(\|u_t\|_{L^2(\R)}^2\right)
		\cdot\E\left(\|v_t\|_{L^2(\R)}^2\right)}\d t.
	\end{equation}
	As was shown in Foondun and Khoshnevisan \cite{FK},
	for every $\lambda>[\lip(\sigma)]^4/(4\theta)$ there exists a finite constant
	$C(\lambda)$ such that
	\begin{equation*}
			\E\left(\|u_t\|_{L^2(\R)}^2\right) \le C(\lambda) \e^{\lambda t}
			\qquad\text{for all $t\ge0$}.
	\end{equation*}
	Therefore, 
	\begin{equation}
		Z(\beta) \le \sqrt{C(\lambda_1)\cdot C(\lambda_2)}
		\int_0^\infty\exp\left( -\left[\beta -\frac{\lambda_1+\lambda_2}{2}\right]
		t\right)\d t<\infty,
	\end{equation}
	for all positive $\lambda_1$ and $\lambda_2$
	such that
	$\lambda_j>[\lip(\sigma_j)]^4/(4\theta)$ for $j=1,2$.

	Consequently, we can deduce from \eqref{recur}
	that as long as $Z(\beta)<\infty$---and this happens
	for all $\beta$ sufficiently large---we have
	\begin{equation}\begin{split}
		Z(\beta) &\le \int_0^\infty \e^{-\beta t}
			\left\< u_0\,, p_{2t}*v_0\right\>_{L^2(\R)} \,\d t+
			\lip(\sigma_1)\lip(\sigma_2)\, r_\beta(0)\, Z(\beta)\\
		&= \mathcal{E}_\beta(u_0\,,v_0) +
			\lip(\sigma_1)\lip(\sigma_2)\, r_\beta(0)\,Z(\beta),
	\end{split}\end{equation}
	since $r_\beta(x) := \int_0^\infty \e^{-\beta t} p_{2t}(x)\,\d t$,
	as can be seen directly from \eqref{r}.

	Thus, whenever $Z(\beta)<\infty$, we have
	\begin{equation}\label{Z(beta):UB}
		Z(\beta) \le  \mathcal{E}_\beta(u_0\,,v_0)
		\left[ 1 -
		 \frac{\lip(\sigma_1)\lip(\sigma_2)}{2\sqrt{\beta\theta}}\right]^{-1}.
	\end{equation}
	An elementary iteration of \eqref{recur} shows that the preceding holds, in fact,
	whenever it makes sense [that is, whenever the right-hand side
	is strictly positive]. Equivalently, that
	\eqref{Z(beta):UB} is valid for all $\beta>(4\theta)^{-1}
	[\lip(\sigma_1)\cdot\lip(\sigma_2)]^2$. The result follows from \eqref{Cov:Z}.
\end{proof}

\section{Proof of the corollaries}
Armed with Theorem \ref{th:main} we conclude the paper by
establishing its corollaries.

\subsection{Proof of Corollary \ref{co:PAM}}
	If $u_0\in L^1(\R)$, then $u_t$ is a.s.\ integrable for all $t>0$;
	see for example Lemma \ref{lem:sufficient}. We work on the more interesting
	converse next.

	Suppose $u_0\in L^\infty(\R)$ is not in $L^1(\R)$; specifically,
	we are assuming that
	\begin{equation}\label{u0:infty}
		\| u_0\|_{L^1(\R)} = \int_{-\infty}^\infty u_0(x)\,\d x=\infty.
	\end{equation}

	Define
	\begin{equation}
		u_{0,N}(x) := u_0(x) \bm{1}_{[-N,N]}(x)\qquad\text{$[x\in\R, N=1,2,\ldots]$},
	\end{equation}
	and let $u_{t,N}(x)$ denote the solution to \eqref{SHE},
	starting from initial function $u_{0,N}$;
	that is,
	\begin{equation}\label{uN}
		u_{t,N}(x) := (p_t* u_{0,N})(x) + \int_{(0,t)\times\R} p_{t-s}(y-x)
		\sigma\left( u_{s,N}(y)\right) \xi(\d s\,\d y).
	\end{equation}

	Since $u_0(x)\ge u_{0,N}(x)$ for all $x\in\R$ and $N\ge 1$,
	an appeal to the comparison theorem of SPDEs \cite{Mueller1,Mueller2} shows that
	$u_t(x)\ge u_{t,N}(x)$ for all $t\ge0$, $x\in\R$, and $N\ge 1$,
	all off a single $\P$-null set. In particular,
	\begin{equation}\label{u:uN}
		\|u_t\|_{L^1(\R)} \ge \| u_{t,N}\|_{L^1(\R)}
		\qquad\text{for all $t\ge0$ and $N\ge 1$},
	\end{equation}
	$\P$-almost surely. It remains to prove that
	\begin{equation}\label{eq:uu}
		\limsup_{N\to\infty}\| u_{t,N}\|_{L^1(\R)}=\infty\quad\text{a.s.}
	\end{equation}

	Because of \eqref{pos}, $u_{t,N}(x)\ge0$ for
	all $t\ge0$, $N\ge 1$, and $x\in\R$ a.s. Therefore,
	$\|u_{t,N}\|_{L^1(\R)}=\int_{-\infty}^\infty u_{t,N}(x)\,\d x$.
	In particular, \eqref{u0:infty},
	\eqref{uN}, and the monotone convergence theorem together
	yield
	\begin{equation}\label{EE}
		\lim_{N\to\infty} \E\left( \|u_{t,N}\|_{L^1(\R)} \right) =
		\lim_{N\to\infty} \|u_{0,N}\|_{L^1(\R)}=\infty.
	\end{equation}
	By Chebyshev's inequality, we obtain \eqref{eq:uu}---%
	and hence the result---once we prove that
	for every $t\ge0$ there exists a finite constant $K(t)$ such that
	\begin{equation}\label{VV}
		\Var\left(\|u_{t,N}\|_{L^1(\R)} \right) \le K(t)\cdot
		\E\left( \|u_{t,N}\|_{L^1(\R)} \right) \qquad\text{for all $N\ge1$}.
	\end{equation}

	According to Theorem \ref{th:main}, there exist
	finite constants $A>0$ and $\beta>0$ such that for all $t\ge0$ and $N\ge 1$,
	\begin{equation}\begin{split}
		\Var\left(\|u_{t,N}\|_{L^1(\R)} \right) &\le A\e^{\beta t}\cdot
			\int_{-\infty}^\infty u_{0,N}(x) (\mathcal{R}_\beta u_{0,N})(x)\,\d x\\
		&\le A\e^{\beta t}\cdot\sup_{x\in\R}(\mathcal{R}_\beta u_0)(x)\cdot
			\E\left( \|u_{t,N}\|_{L^1(\R)}\right),
	\end{split}\end{equation}
	because $u_{0,N}\le u_0$ whence
	$\mathcal{R}_\beta u_{0,N}\le\mathcal{R}_\beta u_0$;  
	eq.\ \eqref{EE} justifies the last line.  Thus
	follows \eqref{VV}, since
	\begin{equation}\label{Ru_0}
		\sup_{y\in\R}(\mathcal{R}_\beta u_0)(y) \le \|u_0\|_{L^\infty(\R)}\cdot
		\int_{-\infty}^\infty r_\beta(x)\,\d x
		= \beta^{-1} \|u_0\|_{L^\infty(\R)}.
	\end{equation}
	The integral was evaluated directly from \eqref{r}.\qed

\subsection{Remarks on initial measures}\label{rem:measure}
	As we mentioned in the Introduction, the results of
	this paper extend to the case that $u_0$ is a measure without
	a great deal of extra effort. Let us point out how this 
	is done for Corollary \ref{co:PAM}.

	\begin{corollary}\label{co:PAM:bis}
		Let $u$ solve \eqref{SHE} for all $x\in\R$, starting from a nonrandom
		Borel measure $u_0$ on $\R$ that satisfies
		\begin{equation}\label{init:meas}
			\int_{-\infty}^\infty \e^{-cx^2} u_0(\d x)<\infty
			\qquad\text{for all } c>0.
		\end{equation}
		Then $u_0$ is finite iff the system \eqref{SHE} is finite.
	\end{corollary}
	
	\begin{proof}[Sketch of Proof]
	The proof is basically an adaptation of the proof of 
	Corollary \ref{co:PAM} to the case that $u_0$ is a measure. We hash out
	the details only where they are not standard.
	
	Theorem 2.4 of Chen and Dalang
	\cite{ChenDalang13Heat} tells us that, under Condition
	\eqref{init:meas}, \eqref{SHE} has a solution $u$
	that is uniquely defined by the following moment condition:
	$\sup_{x\in\R}\E(|u_t(x)|^k)<\infty$ for all $t>0$
	and $k\ge 1$. 
	Standard arguments then imply that the process $\{u_t\}_{t\ge0}$
	is Markov. We now condition at time $t$ and appeal to the 
	preceding moment condition,
	together with the well-known fact that  if we start
	\eqref{SHE} from  an independent,  $L^k(\Omega)$-bounded
	initial function, then the solution is H\"older continuous.
	It follows that $u$ is H\"older continuous a.s.\
	at all points $(t\,,x)$ where $t>0$ and $x\in\R$.
	This comment will take care of all of the meaurability issues that might crop
	up in the remainder of the proof.
	
	If $u_0$ has finite total mass, then
	Theorem 2.4 of \cite{ChenDalang13Heat} 
	and Lemma \ref{lem:sufficient} together imply that 
	$\|u_t\|_{L^1(\R)}<\infty$ a.s.\ for all $t>0$ in exactly
	the same way that the corresponding result held for
	$u_0\in L^1(\R)\cap L^\infty(\R)$.
	Let us, therefore, consider the more interesting case that
	$u_0(\R)=\infty$; of course, \eqref{init:meas} is still in place. We
	propose to prove that $\|u_t\|_{L^1(\R)}=\infty$ a.s.\ for all $t>0$.
	
	For every Borel set $A\subseteq\R$ and all integers $n$ define
	\begin{equation}
	 	v_0(A\cap[n\,,n+1)):= \
		\frac{u_0\left(A\cap [n\,,n+1)\right)}{u_0([n\,,n+1))\vee 1}.
	\end{equation}
	This construction uniquely describes  a Borel measure $v_0$ on $\R$ that is dominated
	by $u_0$ in the sense that $v_0(A)\le u_0(A)$ for all linear Borel sets $A$.
	In particular, $v_0$  satisfies \eqref{init:meas}.
	We observe also that  $v_0$ is  an infinite measure.
	
	Let $v_t(x)$ denote the
	solution to \eqref{SHE}, started at measure $v_0$. The comparison theorem 
	of Chen and Kim \cite{ChenKim14Comp} shows that $v_t(x)\le u_t(x)$
	for all $x\in\R$ and $t>0$ a.s., and hence $\|v_t\|_{L^1(\R)}\le \|u_t\|_{L^1(\R)}$.
	In order to complete our proof 
	we propose to show that $\|v_t\|_{L^1(\R)}=\infty$ a.s.\ for all $t>0$.
	
	Since $v_0([n\,,n+1))\le 1$ and 
	$\sum_{n=-\infty}^\infty r_\beta(n)\le r_\beta(0)+ 2\int_0^\infty r_\beta(x)\,\d x
	=(4\beta\theta)^{-1/2}+2/\beta$,
	the  convexity of $x\mapsto|x|$ shows that
	for all $\beta>0$,
	\begin{equation}
		(\mathcal{R}_\beta v_0)(x) \le 
		\sum_{n=-\infty}^\infty \max_{y\in [n,n+1]}r_\beta(y-x)
		\le (4\beta\theta)^{-1/2}+ 4\beta^{-1},
	\end{equation}
	uniformly for all $x\in\R$.
	In other words, $v_0$ is an infinite measure that 
	has a bounded $\beta$-potential for every $\beta>0$; compare with
	\eqref{Ru_0}. From here on, we can mimic
	our  derivation of Corollary \ref{co:PAM}---starting with
	$v_{0,N}:=$ the restriction of $v_0$ to $[-N\,,N]$---in order to deduce that
	$\|v_t\|_{L^1(\R)}=\infty$ a.s.\ for all $t>0$.  We omit the remaining details.
	\end{proof}
	
\subsection{Proof of Corollary \ref{co:1}}
	Let $\Delta$ denote the Hausdorff  distance between
	the supports of $u_0$ and $v_0$. If $\Delta=0$ then the
	result holds vacuously. Therefore, we will consider only the
	case that $\Delta>0$. In that case,
	\begin{equation}\begin{split}
		\mathcal{E}_\beta(u_0\,,v_0) &= \frac{1}{2\sqrt{\beta\theta}}
			\int_{-\infty}^\infty\d x
			\int_{-\infty}^\infty \d y\
			\e^{-|x-y|\sqrt{\beta/\theta}} u_0(x)v_0(y)\\
		&\le \frac{1}{2\sqrt{\beta\theta}} \|u_0\|_{L^1(\R)}\|v_0\|_{L^1(\R)}
			\cdot\e^{ -\Delta\sqrt{\beta/\theta}}.
	\end{split}\end{equation}
	Consequently, Theorem \ref{th:main} assures us that simultaneously
	for all $t\ge0$ and
	$\beta>[\lip(\sigma_1)\cdot\lip(\sigma_2)]^2/(4\theta)$,
	\begin{equation}
		\Cov\left( \|u_t\|_{L^1(\R)}\,,\|v_t\|_{L^1(\R)}\right)
		\le \frac{K \exp\left(\beta t -\Delta\sqrt{\beta/\theta}\right)}{2\sqrt{\beta\theta} -
		 \lip(\sigma_1)\lip(\sigma_2)}.
	\end{equation}
	where $K:=\lip(\sigma_1)\lip(\sigma_2)\cdot\|u_0\|_{L^1(\R)}\|v_0\|_{L^1(\R)}$.
	Now choose $\beta:= \delta(\Delta/t)^2$ for a suitable choice of
	$\delta>0$ in order to finish.\qed

\subsection{Proof of Corollary \ref{co:2}}
	By Parseval's identity,
	\begin{equation}\begin{split}
		\mathcal{E}_\beta(u_0\,,v_0) &= \frac{1}{2\pi} \int_{-\infty}^\infty
			\widehat{u}_0(z)\overline{\widehat{v}_0(z)} \widehat{r}_\beta(z)\,\d z
			=\frac{1}{2\pi} \int_{-\infty}^\infty
			\frac{\widehat{u}_0(z)\overline{\widehat{v}_0(z)}}{\beta+\theta z^2}\, \d z\\
		&\le\frac{1}{2\pi\beta}\int_{-\infty}^\infty
			\left|\widehat{u}_0(z)\widehat{v}_0(z)\right|\d z\\
		&\le \frac{1}{2\pi\beta}\|u_0\|_{L^1(\R)}\|v_0\|_{L^1(\R)}\cdot
			\text{\rm meas}\left(
			\mathcal{S}[\widehat{u}_0]
			 \cap \mathcal{S}[\widehat{v}_0]\right),
	\end{split}\end{equation}
	since the Fourier transform
	of every function $\varphi\in L^1(\R)$ is bounded uniformly in modulus
	by $\|\varphi\|_{L^1(\R)}$. The corollary follows
	easily from this bound. The remainder of the proof is similar to that of
	Corollary \ref{co:1}, and hence omitted.
	\qed

\begin{small}
\bigskip

\noindent\textbf{Le Chen}, \textbf{Davar Khoshnevisan}, and \textbf{Kunwoo Kim}\\
\noindent Department of Mathematics, University of Utah,
		Salt Lake City, UT 84112-0090\\
\noindent\emph{Emails} \& \emph{URLs}:\\
        \indent\texttt{chen@math.utah.edu}\hfill
        \url{http://www.math.utah.edu/~chen/}
	\indent\texttt{davar@math.utah.edu}\hfill
		\url{http://www.math.utah.edu/~davar/}
	\indent\texttt{kkim@math.utah.edu}\hfill
		\url{http://www.math.utah.edu/~kkim/}
\end{small}


\begin{thebibliography}{999}
%
\bibitem{ChenDalang13Heat}
Chen, Le and Robert C. Dalang.
\newblock Moments and growth indices for nonlinear stochastic heat equation
  with rough initial conditions.
\newblock To appear in {\em Ann.\ Probab.} 
Preprint available at \texttt{arXiv:1307.0600} (2014).
%
\bibitem{ChenKim14Comp}
Chen, Le and Kunwoo Kim.
\newblock On comparison principle and strict positivity of solutions to the nonlinear stochastic fractional heat equations.
\newblock Submitted for publication. Preprint available at \texttt{arXiv:1410.0604} (2014).
%
\bibitem{Dalang} Dalang, Robert C.
	\newblock Extending the martingale measure stochastic integral
		with applications to spatially homogeneous s.p.d.e.'s.
	\newblock {\it Electron.\ J. Probab.}\ {\bf 4}{\it (6)} (1999) 29 pp.\ (electronic).
%
\bibitem{DalangMueller} Dalang, Robert C. and Carl Mueller.
	\newblock Some non-linear S.P.D.E.'s that are second order in time.
	\newblock {\it Electron.\ J. Probab.}\ {\bf 8}{\it (1)} (2003) 21 pp. (electronic).
%
\bibitem{DonohoStark} Donoho, David L. and Philip B. Stark.
	\newblock Uncertainty principles and signal recovery.
	\newblock {\it SIAM J. Appl.\ Math.}\ {\bf 49}{\it (3)} (1989) 906--931.
%
\bibitem{FK} Foondun, Mohammud and Davar Khoshnevisan.
	\newblock On the global maximum of the solution to a stochastic
		heat equation with compact-support initial data.
	\newblock {\it Ann.\ Inst.\ Henri Poincar\'e Probab.\ Stat.}\
		{\bf 46}{\it (4)} (2010) 895--907.
%
\bibitem{Liggett} Liggett, T. M.
	\newblock {\it Interacting Particle Systems}.
	\newblock Springer-Verlag, New York, 1985.
%
\bibitem{Mueller2} Mueller, Carl.
	\newblock {\it Some Tools and Results for Parabolic Stochastic Partial
		Differential Equations} (English summary).
	\newblock In: A Minicourse on Stochastic Partial Differential Equations, 111--144,
	\newblock Lecture Notes in Math.\ {\bf 1962} Springer, Berlin, 2009.
%
\bibitem{Mueller1} Mueller, Carl.
	\newblock On the support of solutions to the heat equation with noise.
	\newblock {\it Stoch.\ \&\ Stoch.\ Rep.}\ {\bf 37}{\it (4)} (1991) 225--245.
%
\bibitem{Spitzer1981} Spitzer, Frank.
	\newblock Infinite systems with locally interacting components.
	\newblock {\it Ann.\ Probab.}\ {\bf 9} (1981) 349--364.
%
\bibitem{Walsh} Walsh, John B.
	\newblock {\it An Introduction to Stochastic Partial Differential Equations}.
	\newblock In: \`Ecole d'\`et\'e de probabilit\'es de
	Saint-Flour, XIV---1984, 265--439.
	Lecture Notes in Math.\ {\bf 1180} Springer, Berlin, 1986.
%
\end{thebibliography}
\end{document}